\documentclass{amsart}
\usepackage[T1]{fontenc}

\usepackage{verbatim}
\usepackage{url}
\usepackage{amsfonts}
\usepackage{cleveref}
\usepackage{amssymb}
\usepackage{mathrsfs}
\usepackage{amsmath}
\usepackage{amsthm}

\newcommand{\lan}{\mathrm{LKan}}
\newcommand{\res}{\mathrm{Res}}
\newcommand{\thy}{\mathrm{Thy}}

\newcommand{\ad}{\mathrm{add}}
\newcommand{\THH}{\mathrm{THH}}

\newcommand{\vectw}{\mathrm{Vect}_{\mathbb{F}_p}^\omega}

\newcommand{\sk}{\mathrm{sk}}

\newcommand{\catp}{\mathrm{Cat}_{\infty}^{\mathrm{poly}}}
\newcommand{\catpk}{\mathrm{Cat}_{\infty, \kappa}^{\mathrm{poly}}}
\newcommand{\catap}{\mathrm{Cat}_{}^{\mathrm{add,poly}}}

\newcommand{\ho}{\mathrm{Ho}}

\newcommand{\stab}{\mathrm{Stab}}

\newcommand{\catex}{\mathrm{Cat}_{\infty}^{\mathrm{perf}}}

\newcommand{\fun}{\mathrm{Fun}}
\newcommand{\catst}{\mathrm{Cat}_\infty^{\mathrm{perf}}}
\newcommand{\perf}{\mathrm{Perf}}
\usepackage{xy}
\input xy

\renewcommand{\sp}{\mathrm{Sp}}
\xyoption{all}

\theoremstyle{definition}
\newtheorem{definition}{Definition}[section]

\newtheorem{remark}[definition]{Remark}
\newtheorem{example}[definition]{Example}
\newtheorem{cons}[definition]{Construction}
\newtheorem{construction}[definition]{Construction}

\theoremstyle{theorem}
\newtheorem{proposition}[definition]{Proposition}
\newtheorem{lemma}[definition]{Lemma}
\newtheorem{corollary}[definition]{Corollary}

\newtheorem{theorem}[definition]{Theorem}

\newcommand{\calA}{\mathcal{A}}
\newcommand{\calB}{\mathcal{B}}
\newcommand{\Cat}{\mathrm{Cat}}
\newcommand{\Add}{\mathrm{add}}

\newcommand{\Hom}{\mathrm{Hom}}
\newcommand{\poly}{\mathrm{poly}}
\newcommand{\Z}{\mathbb{Z}}
\newcommand{\xto}{\xrightarrow}
\newcommand{\Ab}{\mathrm{Ab}}

\newcommand{\proj}{\mathrm{Proj}}
\newcommand{\sym}{\mathrm{Sym}}

\begin{document}

\title{$K$-theory and polynomial functors}
\author{Clark Barwick, Saul Glasman, Akhil Mathew, and Thomas Nikolaus} 
\date{\today}

\maketitle 

\begin{abstract}
We show that the algebraic $K$-theory space of stable $\infty$-categories is
canonically functorial in polynomial functors.  
As a consequence, we obtain a new proof of B\"okstedt's calculation of
$\mathrm{THH}(\mathbb{F}_p)$. 
\end{abstract}

\section{Introduction}

The purpose of this note is to provide an additional structure 
on the higher algebraic $K$-theory of stable $\infty$-categories, arising from polynomial rather
than exact functors. 

In the case of the Grothendieck group $K_0$, 
the construction is due to Dold \cite{Dold72} and Joukhovitski \cite{Jou00}. 
Let $\mathcal{A}$ be an additive category. 
The  group $K_0(\mathcal{A})$ is defined 
to be the group completion of the additive monoid of isomorphism classes of
objects of $\mathcal{A}$. 
By construction, an additive functor $F: \mathcal{A} \to \mathcal{B}$ induces a 
map of abelian groups $K_0(\mathcal{A}) \to K_0(\mathcal{B})$. 

The results of \emph{loc.~cit.} provide additional functoriality on the
construction $K_0$, and  
show that if $F: \mathcal{A}\to \mathcal{B}$ is merely a \emph{polynomial}
functor in the sense of \cite{EM54},  then $F$ nevertheless induces a
canonical map of 
\emph{sets} $F_*: K_0(\mathcal{A}) \to K_0(\mathcal{B})$, such that $F_*$ carries the class
of an object $x \in \mathcal{A}$ to the class of $F(x) \in \mathcal{B}$. 
This polynomial functoriality yields, for example, the $\lambda$-operations on
$K_0(R)$ for a commutative ring $R$, which arise from the exterior power operations on
$R$-modules: the $i$th exterior power functor $\bigwedge^i$ induces a polynomial 
endofunctor on finitely generated projective $R$-modules, and
hence a map of sets $\lambda^i: K_0(R) \to K_0(R)$.  
Here we will extend this polynomial functoriality to higher algebraic $K$-theory. 
To do this, it is convenient to use the setup of the $K$-theory of stable
$\infty$-categories. \\

Let $\mathcal{C}$ be a stable $\infty$-category. As in \cite{BGT, Baruniv}, one constructs
an algebraic $K$-theory space $K(\mathcal{C})$ 
via the Waldhausen $S_\bullet$-construction applied to $\mathcal{C}; $ 
an exact functor $\mathcal{C} \to \mathcal{D}$ of stable $\infty$-categories
induces a map of spaces $K(\mathcal{C}) \to K(\mathcal{D})$. 
For example, when $\mathcal{C} = \perf(X)$ is the stable $\infty$-category
of perfect complexes over a quasi-compact and quasi-separated scheme $X$, 
this is the $K$-theory space of $X$ (introduced in \cite{TT90}; the machinery of
\cite{Qui73a} suffices if $X$
has an ample family of line bundles). 
Moreover, one characterizes \cite{BGT, Baruniv} the construction 
$\mathcal{C} \mapsto K(\mathcal{C})$, when considered as an invariant of all
stable $\infty$-categories and exact functors between them, via a universal
property. 

In this paper, 
we provide additional structure on the construction of algebraic $K$-theory in
analogy with the results on $K_0$ from \cite{Dold72, Jou00}, and characterize it by the same
universal property. 

To formulate the result, 
let $\catex$
denote the $\infty$-category of small, idempotent-complete stable $\infty$-categories and exact functors
between them, and let $\mathcal{S}$ be the $\infty$-category of spaces. 
Algebraic $K$-theory defines a functor $$K: \catex \to
\mathcal{S}.$$ It receives a natural transformation from the  functor $\iota: \catex \to \mathcal{S}$ which carries
$\mathcal{C} \in \catex$ to the space 
of objects in $\mathcal{C}$, i.e., we have a map $\iota \to K$ of functors
$\catex \to \mathcal{S}$. 
The universal property of $K$-theory 
\cite{BGT, Baruniv} 
states that $K$ is the initial functor  
$\catex \to \mathcal{S}$ receiving a map from $\iota$ such that $K$ preserves
finite products, splits semiorthogonal decompositions, and is grouplike.

Let $\mathcal{C}, \mathcal{D}$ be small, stable idempotent-complete
$\infty$-categories. A functor $f: \mathcal{C} \to \mathcal{D}$ is said
to be \emph{polynomial} if it is $n$-excisive for some $n$ in the sense of
\cite{Goo92}. 
Let 
$\catp$ 
denote the $\infty$-category 
of small, idempotent-complete stable $\infty$-categories and
{polynomial} functors between them. Thus, we have an inclusion $\catex \to \catp$; note that $\catex, \catp$ have
the same objects, but $\catp$ has many more morphisms. 
Our main result states that $K$-theory can be defined on $\catp$. 

\begin{theorem} 
\label{mainthm}
There is a 
canonical extension of the functor $K: \catex \to \mathcal{S}$ to a functor
$\widetilde{K}: \catp \to \mathcal{S}$. 
\end{theorem} 

In fact, the construction $\widetilde{K}$ is characterized by a similar universal
property. Namely, one has a canonical extension of the functor $\iota: \catex
\to \mathcal{S}$ to a functor $\iota:
\catp \to \mathcal{S}$ since $\iota$ can be defined on all $\infty$-categories
and functors between them. 
One defines the functor $\widetilde{K}$ (with a natural map $\iota \to
\widetilde{ K}$) by enforcing the same
universal property on $\catp$. The main computation one then carries out is that
$\widetilde{K}$ restricts to $K$ on $\catex$, i.e., one recovers the original
$K$-theory functor.

\begin{remark} 
\Cref{mainthm}, together with the theory of the Bousfield--Kuhn functor
\cite{Kuh89, Bou01},
 implies that for $n \geq 1$, the (telescopic) $T(n)$-localization of the
algebraic $K$-theory \emph{spectrum} of a stable $\infty$-category is functorial
in polynomial functors, i.e., extends to $\catp$. 
\end{remark} 

As an application of 
\Cref{mainthm}, we give a new proof of B\"okstedt's calculation of
$\THH(\mathbb{F}_p)$. 
Our insight is that B\"okstedt's calculation 
is equivalent by \cite{AN21} to the statement that $\mathrm{TR}(\mathbb{F}_p)$ is
discrete. Using arguments similar to \cite{Hi81, Kra80} (which show that the
$p$-adic $K$-theory of any perfect $\mathbb{F}_p$-algebra vanishes in positive degrees) and the connection
between $\mathrm{TR}$ and cyclic $K$-theory \cite{LM12}, one can prove the
desired discreteness of $\mathrm{TR}(\mathbb{F}_p)$.

\subsection*{Motivation and related work} 
Many previous authors have considered various types of non-additive operations
on algebraic $K$-theory spaces, which provided substantial motivation for this work.

An important example is given by operations in the $K$-theory space (and on the
$K$-groups) of a ring $R$ arising from exterior and symmetric
power functors on $R$-modules. 
Constructions of such maps  appear in many sources, including \cite{Hi81, Kra80,
So85, Gra89, Ne91, Lev97, HKT}. 
We refer to \cite{KoeckZanchetta} for some comparisons
between these constructions (including ours). 
Another example in this vein is given by the multiplicative norm maps along
finite \'etale maps constructed in \cite{BH17}. 

A different instance of non-additive operations in $K$-theory arises in Segal's 
approach to the Kahn--Priddy theorem \cite{Seg74}. These maps arise from the
$K$-theory of non-additive categories (such as the category of finite sets) and
cannot be obtained from \Cref{mainthm}.

\subsection*{Notation and conventions}
We freely use the language of $\infty$-categories and higher algebra as in
\cite{Lur09, HA}. 
Throughout, we let $\mathcal{S}$ denote the $\infty$-category of spaces, and
$\sp$ the $\infty$-category of spectra. 

\subsection*{Acknowledgments}

We are very grateful to Bhargav Bhatt, Lukas Brantner,  Dustin Clausen, Rosona
Eldred, Matthias Flach, Lars Hesselholt, Marc Hoyois, Jacob Lurie and  Peter Scholze for helpful
discussions. We also thank Benjamin Antieau for several comments and
corrections on an earlier
draft. 

The third author was supported by the NSF Graduate Fellowship under grant
DGE-114415 as this work was begun and was a Clay Research Fellow when this work
was completed. The third author also thanks the University of Minnesota
and the Hausdorff Institute of
Mathematics (during the fall 2016 Junior Trimester program ``Topology'') 
for their hospitality.  The fourth author was funded by the Deutsche
Forschungsgemeinschaft (DFG, German Research Foundation) under Germany's Excellence Strategy EXC 2044 390685587, Mathematics M\"unster: Dynamics-Geometry-Structure.

\section{Polynomial functors}

\subsection{Simplicial and filtered objects}
In this subsection, we review basic facts about simplicial objects in a stable
$\infty$-category. In particular, we review the stable version of the Dold-Kan correspondence, following Lurie
\cite{HA}, which connects simplicial and filtered objects.

To begin with, we review the classical Dold-Kan correspondence.
A general reference for this is \cite[III.2]{GoJ99} for the category of abelian
groups or \cite[8.4]{Wei94} for an abelian category. We refer to \cite[1.2.3]{HA} for a treatment for arbitrary additive
categories. 
\begin{theorem}[Dold-Kan correspondence] 
Let $\mathcal{A}$ be an additive category which is idempotent-complete. 
Then we have an equivalence 
of categories
\[ \fun(\Delta^{op}, \mathcal{A}) \simeq \mathrm{Ch}_{\geq 0}(\mathcal{A}),  \]
between the category
$\fun(\Delta^{op}, \mathcal{A})$ of simplicial objects in $\mathcal{A}$ and the
category 
$\mathrm{Ch}_{\geq 0}(\mathcal{A})$ of nonnegatively graded chain complexes in
$\mathcal{A}$. 
\end{theorem} 
The Dold-Kan equivalence arises as follows. Given a simplicial object $X_\bullet \in 
\fun(\Delta^{op}, \mathcal{A})$, we form an associated chain complex
$C_\ast$ such that: 
\begin{enumerate}
\item $C_n$ is a direct summand of $X_n$, and is given by the intersection of
the kernels $\bigcap_{i \geq 1}\mathrm{ker}(d_i) $ where the $d_i$'s give the
face maps $X_n \to X_{n-1}$. If $\mathcal{A}$ is only
assumed additive, the existence of this kernel is not a priori obvious (cf.
\cite[Rmk. 1.2.3.15]{HA}). 
However, we emphasize that the object $C_n$ depends only on the face maps $d_i,
i \geq 1$.
\item The differential $C_n \to C_{n-1}$ comes from the face map $d_0$ in the
simplicial structure. 
\end{enumerate}

The Dold-Kan correspondence has  an analog for stable
$\infty$-categories, formulated in \cite[Sec.~1.2.4]{HA}, yielding a  correspondence between simplicial and \emph{filtered}
objects.

\begin{theorem}[{Lurie \cite[Th. 1.2.4.1]{HA}}] 
\label{nonabelianDK}
Let $\mathcal{C}$ be a stable $\infty$-category. Then we have an equivalence
of stable $\infty$-categories
\[ \fun(\Delta^{op}, \mathcal{C} ) \simeq \fun( N \mathbb{Z}_{\geq 0},
\mathcal{C}),   \]
which sends a simplicial object $X_\bullet \in\fun(\Delta^{op}, \mathcal{C} )$
to the filtered object $| \sk_0 X_\bullet|  \to |\sk_1 X_\bullet| \to |\sk_2
X_\bullet|  \to \dots $. 
\end{theorem}

\begin{remark}[Making Dold-Kan explicit]
We will need to unwind the 
correspondence as follows. 
Given a filtered object $Y_0 \to Y_1 \to Y_2 \to \dots $ in $\mathcal{C}$, 
we form the sequence of cofibers $Y_0, Y_1/Y_0, Y_2/Y_1, \dots$, i.e., the
associated graded. We have boundary maps
$$Y_1/Y_0 \to \Sigma Y_0, \quad Y_2/Y_1 \to \Sigma Y_1/Y_0, \dots ,$$ and the sequence 
\begin{equation} \label{chcxhoC}  \dots \to \Sigma^{-2}Y_2/Y_1 \to \Sigma^{-1}Y_1/Y_0 \to Y_0  \end{equation}
forms a chain complex in the homotopy category $\ho(\mathcal{C})$: the
composite of any two successive maps is nullhomotopic. 
If $\mathcal{C}$ is an idempotent-complete stable $\infty$-category, then
$\ho(\mathcal{C})$ is an idempotent-complete additive category. Given 
a simplicial object $X_\bullet \in \fun(\Delta^{op}, \mathcal{C})$, we can also
extract a simplicial object in $\ho(\mathcal{C})$ and thus a chain complex in
$\ho(\mathcal{C})$ by the classical Dold-Kan correspondence. A basic compatibility states that this produces the sequence 
\eqref{chcxhoC} (for the corresponding filtered object $Y_0 \to Y_1 \to \dots$), i.e., the additive and stable Dold-Kan correspondences are
compatible \cite[Rem. 1.2.4.3]{HA}. 

\end{remark}

\subsection{Polynomial functors of stable $\infty$-categories}
In this subsection, we review
the notion of polynomial functor between stable $\infty$-categories.
We first discuss the analogous
notion for additive $\infty$-categories.

\begin{definition}[Eilenberg-MacLane \cite{EM54}]
Let $F: \calA \to \calB$ be a functor between additive $\infty$-categories and assume
first that $\calB$ is idempotent complete. Then: 
\begin{itemize}
\item  
$F$ is called \emph{polynomial} of degree $\leq -1$ if it is the trivial functor which sends everything to $0$. 
\item
$F$ is called \emph{polynomial} of degree $ \leq 0$ if it is a constant functor. 
\item Inductively, $F$ is called polynomial of degree $\leq n$ if for each $Y \in \calA$ the functor
$$
(D_YF)(X) := \mathrm{ker}\Big(F(Y \oplus X) \xto{F(\mathrm{pr}_X) }F(X)\Big)
$$
is polynomial of degree at most $n-1$.
 Note that this kernel exists as we have assumed $\calB$ to be idempotent complete and it is the complementary summand to $F(X)$.

\end{itemize}
We let $\fun_{\leq n}(\mathcal{A}, \mathcal{B}) \subset \fun(\mathcal{A},
\mathcal{B})$ be the subcategory spanned by functors of degree $\leq
n$.\footnote{This notion easily generalizes when $\mathcal{B}$ is not idempotent complete. 
If $\calB$ is not idempotent complete then $F: \calA \to \calB$ is called
polynomial of degree $\leq n$ if the composition $\calA \to \calB \to
\overline{\calB}$ is polynomial of degree $\leq n$, where $\calB \to
\overline{\calB}$ is an idempotent completion.}
Finally, a functor $F: \calA \to \calB$ is \emph{polynomial} if it is polynomial
of degree $\leq n$ for some $n$. 
\end{definition}

The notion of a polynomial functor behaves in a very intuitive fashion. 
For example, the composite of a functor of degree $\leq m$ with one of degree
$\leq n$ is of degree $\leq mn$. As another example, we have: 
\begin{example} 
Let $R$ be a commutative ring.
The symmetric power functors $\sym^i$ and exterior power functors $\bigwedge^i$ on
the category $\proj_R^\omega$ of finitely generated projective $R$-modules are polynomial of degree $\leq i$. 
\end{example}

Let $\mathcal{C} $ be a stable $\infty$-category. 

\begin{definition}[$n$-skeletal simplicial objects]
We say that a simplicial object $X_\bullet \in \fun(\Delta^{op}, \mathcal{C})$
is \emph{$n$-skeletal} if
it is left Kan extended from its restriction to $\Delta_{\leq n} \subset
\Delta$. \end{definition}

\begin{remark}[$n$-skeletal geometric realizations exist] 
Note that if $X_\bullet$ is $n$-skeletal for some $n$, then the
geometric realization $|X_\bullet|$ exists in $\mathcal{C}$. 
\end{remark}

\begin{remark}[$n$-skeletal objects via Dold-Kan]
\label{nskeletalDK}
The condition that $X_\bullet$ should be $n$-skeletal depends only on the
underlying homotopy category of $\mathcal{C}$, considered as an additive
category. 
Namely, using the Dold-Kan correspondence, we can form a chain complex $C_{\ast}$ in the homotopy category of
$\mathcal{C}$ from $X_\bullet$, and then we claim that $X_\bullet$ is
$n$-skeletal if and only if $C_{\ast}$ vanishes for $\ast > n$. 

Indeed, if $X_\bullet$ is $n$-skeletal, then clearly the maps $|\sk_i X_\bullet|
\to |\sk_{i+1} X_\bullet|$ are equivalences for $i \geq n$, which implies that
$C_{\ast} = 0$ for $\ast > n$. 
Conversely, if $C_{\ast} = 0$ for $\ast > n$, then the map of simplicial
objects $\sk_n X_\bullet \to
X_\bullet$
has the property that it induces an equivalence on $i$-truncated geometric
realizations for all $i \in \mathbb{Z}_{\geq 0}$, which implies by Theorem~\ref{nonabelianDK} that it is an equivalence of simplicial objects. 
\end{remark} 

\begin{definition} 

We say that a functor $f: \mathcal{C} \to \mathcal{D}$ of stable
$\infty$-categories preserves \emph{finite
geometric
realizations} if for every simplicial object $X_\bullet $ in $\mathcal{C}$ which
is $n$-skeletal for some $n$, the colimit $|F(X_\bullet)|$ exists and the
natural map $|F(X_\bullet)| \to F(|X_\bullet|)$ is an equivalence in
$\mathcal{C}$. 
\end{definition}

\begin{proposition} 
Let $\mathcal{C}, \mathcal{D}$ be stable $\infty$-categories.
Let $F: \mathcal{C} \to \mathcal{D}$ be a functor such that the underlying
functor on additive categories $\ho(\mathcal{C}) \to \ho(\mathcal{D})$ is
polynomial of degree $\leq d$. Then if $X_\bullet $ is an $n$-skeletal
simplicial object in $\mathcal{C}$, $F(X_\bullet)$ is an $nd$-skeletal
simplicial object in $\mathcal{D}$. 
\end{proposition} 
\begin{proof} 
As above in Remark~\ref{nskeletalDK}, this is a statement purely at the level of homotopy categories. 
That is, the simplicial object $X_\bullet$ defines a simplicial object
 of the
additive category $\mathrm{Ho}(\mathcal{C})$, and thus a chain complex 
$C_\ast$
in
$\mathrm{Ho}(\mathcal{C})$. Similarly $F(X_\bullet)$ defines a simplicial object
of $\mathrm{Ho}(\mathcal{D})$ 
and thus a chain complex $D_\ast$ of $\mathrm{Ho}(\mathcal{D})$. The claim is that if 
$C_\ast = 0 $ for $\ast > n$ then $D_\ast = 0$ for $\ast > nd$. 
This is purely a statement about additive categories. 
For proofs, see \cite[Lem. 3.3]{GS} or \cite[4.23]{DP}. 
\end{proof}

\begin{definition}[Polynomial functors]
\label{polynomialdef}
Let $\mathcal{C}, \mathcal{D}$ be stable $\infty$-categories. 
We say that a functor $F: \mathcal{C} \to \mathcal{D}$ is \emph{polynomial of
degree $\leq d$} if the underlying functor $\ho(F) : \ho(\mathcal{C}) \to
\ho(\mathcal{D})$ of additive categories is polynomial 
of degree $\leq d$ and if $F$ preserves finite geometric realizations. 
We let $\fun_{\leq d}(\mathcal{C}, \mathcal{D}) \subset \fun(\mathcal{C},
\mathcal{D})$ denote the full subcategory spanned by functors of degree $\leq
d$. 
\end{definition} 

This notion of a polynomial functor will be fundamental to this paper. 
The definition is equivalent to the more classical definition of a polynomial
functor, due to
Goodwillie \cite{Goo92}, via $n$-excisivity, defined more generally outside the
stable context. 
For convenience, we describe the comparison below. 
\begin{definition}

Let $[n] = \left\{0, 1, \dots, n\right\}$, and let $\mathcal{P}([n])$ denote the
nerve of the poset of subsets of $[n]$, so that $\mathcal{P}([n]) \simeq
(\Delta^1)^{n+1}$. 
An \emph{$(n+1)$-cube} in $\mathcal{C}$ is a functor $f:\mathcal{P}([n]) \to
\mathcal{C}$. The $(n+1)$-cube is said to be \emph{strongly coCartesian} if it is 
left Kan extended from the subset $\mathcal{P}_{\leq 1}([n]) \subset
\mathcal{P}([n])$ spanned by subsets of cardinality $\leq 1$, and
\emph{coCartesian} 
if it is a colimit diagram. Note that a diagram is coCartesian if and only if
it is a limit or \emph{Cartesian} diagram by \cite[Prop.
1.2.4.13]{HA}. 
\end{definition}

\begin{definition}[{Goodwillie \cite[Def. 3.1]{Goo92}}] 
Let $\mathcal{C}, \mathcal{D}$ be small stable $\infty$-categories and let $F:
\mathcal{C} \to \mathcal{D}$ be a functor. For $n \geq 0$, we say that $F$ is
\emph{$n$-excisive} if $F$ carries strongly coCartesian $(n+1)$-cubes to coCartesian
$(n+1)$-cubes. \end{definition}

\begin{example} 
For $n  = 0$, the condition is that $F$ should be a constant functor. 
For $n = -1$, we say that $F$ is $n$-excisive if $F$ is identically
zero.  For $n =1$, $F$ is $1$-excisive (or simply excisive) if it carries
pushouts to pullbacks. 

\end{example}

\begin{proposition} 
The functor $F: \mathcal{C} \to \mathcal{D}$ is polynomial of degree $\leq n$
(in the sense of Definition~\ref{polynomialdef})  if and
only if it is $n$-excisive. 
\end{proposition} 
\begin{proof} 
Suppose first that $F$ is polynomial of degree $\leq n$, and fix 
a strongly coCartesian $(n+1)$-cube in $\mathcal{C}$, which is determined by a
collection of maps $X \to Y_i$, $i = 0, 1, \dots, n$, in $\mathcal{C}.$ We
would like to show that $F$ carries this $(n+1)$-cube to a coCartesian one. Since
every object in $\fun(\Delta^1, \mathcal{C})$ is a finite geometric realization 
of arrows of the form $A \to A \oplus B$, and $F$ preserves finite 
geometric realizations, we may assume that we have $Y_i
\simeq X \oplus Z_i$ for objects $Z_i, i = 0, 1, \dots, n$. 

Denote the above strongly coCartesian $(n+1)$-cube by $ f: \mathcal{P}([n]) \to
\mathcal{C}$. Let $\mathcal{P}'( [n]) \subset \mathcal{P}([n])$ be the
complement of the terminal object. 
The cofiber
$\varinjlim_{\mathcal{P}'([n])} F \circ f \to F ( f([n]))$ 
is given precisely by the iterated derivative $D_{Z_0} D_{Z_1} \dots D_{Z_n} F
(X)$, as follows easily by induction on $n$. 
Thus, if $F$ is polynomial of degree $\leq n$, we see that 
$F$ is $n$-excisive. 

Conversely, if $F$ is $n$-excisive, the previous paragraph shows that $\ho(F)$
is polynomial of degree $\leq n$. It remains to show that $F$ preserves finite
geometric realizations. 
This follows from the general theory and classification of $n$-excisive
functors, cf.~also \cite[Prop.~3.36]{BM19} for an account. 
We can form the embedding $\mathcal{D} \hookrightarrow
\mathrm{Ind}(\mathcal{D})$ to replace $\mathcal{D}$ by a presentable stable
$\infty$-category; this inclusion preserves finite geometric realizations. In
this case, the general theory shows that $F$ can be built up via a finite
filtration from its
\emph{homogeneous layers}, and each homogeneous layer 
of degree $i$
is of the form $X \mapsto B(X, X, \dots, X)_{h \Sigma_i}$ for $B: \mathcal{C}^i
\to \mathrm{Ind}(\mathcal{D})$ a functor which is exact in each variable and symmetric in its
variables (see \cite[Sec. 6.1]{HA} for a reference in this setting). 
Therefore, it suffices to show that if $B: \mathcal{C}^i \to
\mathrm{Ind}(\mathcal{D})$ is a functor which is exact in each variable, then
$X \mapsto B(X, X, \dots, X)$  preserves finite geometric realizations. 
This now follows from the cofinality of the diagonal $\Delta^{op} \to
(\Delta^{op})^i$ and the fact that
$B$ preserves finite geometric realizations in each variable separately, since
it is exact. 
\end{proof}

\subsection{Construction of polynomial functors}
We now discuss some examples of polynomial functors on stable
$\infty$-categories; these will arise from the derived functors of polynomial
functors of additive $\infty$-categories. 
We first review the relationship between additive, prestable,  and stable
$\infty$-categories. Compare \cite[Sec. C.1.5]{SAG}.

\begin{cons}[The stable envelope] 
Given any small additive $\infty$-category $\mathcal{A}$, there is a universal stable
$\infty$-category $\stab(\mathcal{A})$ equipped with an additive, fully faithful functor
$\mathcal{A} \to \stab(\mathcal{A})$. 
Given any small stable $\infty$-category $\mathcal{B}$, any additive functor
$\mathcal{A} \to \mathcal{B}$ canonically extends to an exact functor
$\stab(\mathcal{\mathcal{A}}) \to \mathcal{B}$. 
In other words, $\stab$ is a left adjoint from the natural forgetful
functor from the $\infty$-category of small stable $\infty$-categories to the
$\infty$-category of small additive $\infty$-categories.
We refer to $\stab(\mathcal{A})$ as the \emph{stable envelope} of $\mathcal{A}$. 

Explicitly, $\stab(\mathcal{A})$ is the stable subcategory of
the $\infty$-category
$\fun^{\times}(\mathcal{A}^{op}, \sp)$ of finitely product-preserving
presheaves of spectra on $\mathcal{A}$ generated by the image of the Yoneda
embedding. 
\end{cons} 

\begin{cons}[The prestable envelope $\stab(\mathcal{A})_{\geq 0}$]
Let $\mathcal{A}$ be a small additive  $\infty$-category as above, and let
$\stab(\mathcal{A})$ be its stable envelope. 
We let $\stab(\mathcal{A})_{\geq 0}$ 
denote the 
subcategory of the nonabelian derived $\infty$-category 
\cite[Sec.~5.5.8]{Lur09}
$\mathcal{P}_{\Sigma}(\mathcal{A})$ generated under finite colimits 
by $\mathcal{A}$. Then $\stab(\mathcal{A})_{\geq 0}$ is a prestable
$\infty$-category \cite[Appendix C]{SAG} and is the universal prestable
$\infty$-category receiving an additive functor $\mathcal{A} \to
\mathcal{P}_{\Sigma}(\mathcal{A})$, which we call the \emph{prestable envelope}
of $\mathcal{A}$. There are fully faithful embeddings
\[ \mathcal{A} \subset \stab(\mathcal{A})_{\geq 0} \subset \stab(\mathcal{A}),  \]
and $\stab(\mathcal{A})$ is also the stabilization of the prestable
$\infty$-category $\stab(\mathcal{A})_{\geq 0}$. 
\end{cons}

\begin{example} 
Let $R$  be a ring, or more generally a connective $\mathbb{E}_1$-ring spectrum. 
Then one has a natural additive $\infty$-category $\proj_R^\omega$ of finitely
generated, projective $R$-modules. The stable envelope is given by the
$\infty$-category $\perf(R)$ of perfect $R$-modules, and the prestable envelope
is given by the $\infty$-category $\perf(R)_{\geq 0}$ of connective perfect
$R$-modules. 
\end{example} 
Our main result is the following extension principle, which allows one to
extend polynomial functors from an additive $\infty$-category to the stable
envelope. 
\begin{theorem}[Extension of polynomial functors]
\label{extendpolyfunctor}
 Let $\mathcal{D}$ be a stable $\infty$-category and $\mathcal{A}$ be an additive $\infty$-category.
Pullback along the functor $\mathcal{A} \to \stab(\mathcal{A})$ induces an
equivalence of $\infty$-categories
\[ \fun_{\leq n}(\mathcal{A}, \mathcal{D}) \simeq \fun_{\leq
n}(\stab(\mathcal{A}), \mathcal{D}),  \]
between degree $\leq n$ functors $\mathcal{A} \to \mathcal{D}$ (in the sense of
additive $\infty$-categories) and degree $\leq n$ functors (in the sense of
stable $\infty$-categories) $\stab(\mathcal{A}) \to \mathcal{D}$. 
\end{theorem}

We use the following crucial observation 
due to Brantner.  

\begin{lemma}[{Extending from the connective objects,
cf.~\cite[Th.~3.35]{BM19}}] 
\label{extendfromconnective}
Let $\mathcal{D}$ be a presentable, stable $\infty$-category. 
Restriction induces an equivalence
of $\infty$-categories
\[ \fun_{\leq n}( \stab(\mathcal{A})_{\geq 0}, \mathcal{D}) \simeq
\fun_{\leq n}(\stab(\mathcal{A}), \mathcal{D}),
\]
between $n$-excisive functors $\stab(\mathcal{A})_{\geq 0} \to \mathcal{D}$ and
 $n$-excisive functors
$\stab(\mathcal{A}) \to \mathcal{D}$. 
\end{lemma} 

\begin{lemma} 
\label{additivepolythick}
Let $\mathcal{C}, \mathcal{D}$ be stable $\infty$-categories. 
Let $\mathcal{A} \subset \mathcal{C}$ be an additive 
subcategory which generates $\mathcal{C}$ as a stable subcategory. 
Let $F: \mathcal{C} \to \mathcal{D}$ be a degree $\leq n$ functor. Suppose
$F|_{\mathcal{A}}$ has image contained in a stable subcategory $\mathcal{D}'
\subset \mathcal{D}$. Then $F$ has image contained in $\mathcal{D}'$.
\end{lemma} 
\begin{proof} 
We use essentially the notions of \emph{levels} in $\mathcal{C}$, cf. \cite[Sec.
2]{ABIM}, although we are not assuming idempotent completeness. 

We define an increasing and exhaustive filtration
of subcategories
$\mathcal{C}_{\leq 1} \subset \mathcal{C}_{\leq 2} \subset \dots \subset
\mathcal{C}$ as follows. 
The subcategory $\mathcal{C}_{\leq 1}$ is the additive closure of $\mathcal{A}$
under shifts, so any object of $\mathcal{C}_{\leq 1}$ can be written in the form 
$\Sigma^{i_1} A_1 \oplus \dots \oplus \Sigma^{i_n} A_n$ for some 
$A_1, \dots, A_n \in \mathcal{A}$ and $i_1, \dots, i_n \in \mathbb{Z}$. 
Inductively, we let $\mathcal{C}_{\leq n}$ denote the subcategory of objects
that are extensions of objects in $\mathcal{C}_{\leq a} $ and $\mathcal{C}_{\leq
b}$ for $0 < a,b < n$ with $a + b \leq n$. 
Each $\mathcal{C}_{\leq n}$ is closed under translates, and it is easy to see
that $\bigcup_{n} \mathcal{C}_{\leq n}$ is stable and contains $\mathcal{A}$ and
hence equals $\mathcal{C}$. 

We claim first that $F(\mathcal{C}_{\leq 1}) \subset \mathcal{D}$. 
That is, we need to show that 
for $A_1, \dots, A_n \in \mathcal{A}$ and $i_1, \dots, i_n \in \mathbb{Z}$, we
have $F(\Sigma^{i_1} A_1 \oplus \dots \oplus \Sigma^{i_n} A_n) \in \mathcal{D}$. 
If $i_1, \dots, i_n \geq 0$, then we can choose a simplicial object in
$\mathcal{A}$ which is $d$-truncated for some $d$ and whose geometric
realization is 
$\Sigma^{i_1} A_1 \oplus \dots \oplus \Sigma^{i_n} A_n$. Since $F$ preserves
finite geometric realizations, the claim follows. 
In general, for any object $X \in \mathcal{C}$, we can recover $F(X)$ as a
finite homotopy limit of $F(0), F(\Sigma X), F(\Sigma X \oplus \Sigma X ),
\dots$ via the $T_n$-construction, since $F$ is $n$-excisive. Using this, we
can remove the hypotheses that $i_1, \dots, i_n \geq 0$ and conclude that
$F(\mathcal{C}_{\leq 1}) \subset \mathcal{D}.$

Given an object of $\mathcal{C}_{\leq n}$ with $n > 1$, 
it is  an extension of objects in
$\mathcal{C}_{\leq a}$ and $\mathcal{C}_{\leq b}$ for some $a, b < n$. 
In view of Construction~\ref{cechnerve} below, 
it follows that  any such object can be written as
a finite geometric realization of objects of level $< n$. 
Since $F$ preserves finite geometric realizations, it follows by induction that
$F(\mathcal{C}_{\leq n}) \subset \mathcal{D}$. 
\end{proof}

\begin{proof}[Proof of \Cref{extendpolyfunctor}] 
Embedding $\mathcal{D}$ inside $\mathrm{Ind}(\mathcal{D})$, we may assume that
$\mathcal{D}$ is actually presentable. 
By Lemma~\ref{additivepolythick}, we do not lose any generality by doing so. 

By the universal property of $\mathcal{P}_\Sigma$, we have an equivalence 
\begin{equation} \label{cats1}
\fun_{\leq n}(\mathcal{A}, \mathrm{Ind}(\mathcal{D})) \simeq
\fun_{\leq n}^{\Sigma}(\mathcal{P}_{\Sigma}(\mathcal{A}),
\mathrm{Ind}(\mathcal{D}))
,  
\end{equation}
where $\Sigma$ denotes functors which preserve sifted colimits; the universal
property gives this without the $\leq n$ condition, which we can then impose. 
Now 
the inclusion
$\stab(\mathcal{A})_{\geq 0} \subset \mathcal{P}_{\Sigma}(\mathcal{A})$ 
exhibits the target as the $\mathrm{Ind}$-completion of the source,
which yields an equivalence \begin{equation} \label{cats2}  \fun_{\leq n}(
\stab(\mathcal{A})_{\geq 0}, \mathrm{Ind}(\mathcal{D})) \simeq
\fun_{\leq n}^\omega( \mathcal{P}_\Sigma(\mathcal{A}),
\mathrm{Ind}(\mathcal{D})),
\end{equation}
where $\omega$ denotes functors which preserve filtered colimits. 
By definition, any functor in 
$\fun_{\leq n}^\omega( \mathcal{P}_\Sigma(\mathcal{A}),
\mathrm{Ind}(\mathcal{D}))$ preserves finite geometric realizations, and hence
all geometric realizations; thus it also preserves all sifted colimits. 
This shows that the categories in \eqref{cats1} and \eqref{cats2} are
identified. 
Now the result follows by combining this identification with \Cref{extendfromconnective}. 
\end{proof}

\begin{corollary} 
Let $\mathcal{A} \to \mathcal{B}$ be additive $\infty$-categories. 
Then a degree $\leq n$ functor $\mathcal{A} \to \mathcal{B}$ canonically prolongs to
a degree $\leq n$ functor of stable $\infty$-categories, $\stab(\mathcal{A}) \to \stab(\mathcal{B})$. \qed
\end{corollary} 

\begin{example} 
Let $R$ be a commutative ring. Then we have a functor $\sym^i: \perf(R) \to
\perf(R)$ which is $i$-excisive and which extends the usual symmetric powers of
finitely generated projective $R$-modules. 
We can regard this as a derived functor of the usual symmetric power, although
we are allowing nonconnective objects as well. 
\end{example}

The above construction of extending functors, for $\stab(\mathcal{A})_{\geq 0}$, 
is the classical one of Dold-Puppe \cite{DP} of ``nonabelian derived functors''
constructed using simplicial resolutions. Compare also \cite{JM97} for the
connection between polynomial functors on additive categories and $n$-excisive
functors.  
The extension to $\stab(\mathcal{A})$, at least in certain cases, goes back to Illusie
\cite[Sec. I-4]{Illusie}
in work on the cotangent complex, using simplicial cosimplicial objects.

\section{$K_0$ and polynomial functors}

\subsection{Additive $\infty$-categories}
In this section, we review the result 
of \cite{Dold72, Jou00} that
$K_0$ of additive $\infty$-categories is naturally
functorial in polynomial functors; this special case of \Cref{mainthm} will play
an essential role in its proof.

\begin{definition}[Passi \cite{Passi}]
\label{degreenmapgroup}
Let $M$ be an abelian monoid and $A$ be an abelian group. We will define
inductively when a map $f: M \to A$ (of sets) is called polynomial of degree $\leq n$.  \begin{itemize}
\item  
A map $f$ is called \emph{polynomial of degree} $\leq -1$ if it is identically
zero.
\item
A map $f$  if called \emph{polynomial of degree $\leq n$} if for each $y \in M$ the map
$D_y f: M \to A$ defined by 
$$
(D_y f)(x) := f(x+y) - f(x)
$$
is polynomial of degree $\leq n-1$.
\end{itemize}
 We say that $f$ is \emph{polynomial} if it is polynomial of degree $n$ for some $n$.

We denote the set of polynomial maps $M \to A$ of degree $\leq n$ by $\Hom_{\leq n}(M,A)$. It is straightforward to check that composing polynomial maps whenever this is defined is again polynomial and changes the degree in the obvious way.
\end{definition}

\begin{example}
A  map $f: \Z \to \Z$ is polynomial of degree $\leq n$ precisely if it  can be represented by a polynomial of degree $n$ with rational coefficients. In this case it has to be of the form 
$$f(x) = \sum_{i = 0}^n \alpha_i {{x}\choose{i}}$$
with $\alpha_i \in \Z$.
\end{example}

Now let $i: M \to M^+$ be the group completion of the abelian monoid $M$. Then
the following result states that we can always extend polynomial maps uniquely over the group completion. This is surprising if one thinks about how to extend to a formal difference.

\begin{theorem}[{\cite[Prop.~1.6]{Jou00}}]\label{thm_passi}
For any abelian monoid $M$ and abelian group $A$, 
the map 
$$
i ^*: \Hom_{\leq n}(M^+,A) \to \Hom_{\leq n}(M,A) 
$$
is a bijection. \end{theorem}
The 
proof in \emph{loc.~cit.} gives an explicit argument. 
For the convenience of the reader, we include an abstract argument via monoid
and group rings. 
\begin{proof}
The first step is to reformulate the condition for a map $f: M \to A$ to be polynomial of degree $\leq n$. To do this we will temporarily for this proof write $M$ multiplicatively, in particular $1 \in M$ is the neutral element. 

A map of sets $f: M \to A$ is polynomial of degree at most $n$ precisely if the induced map
$\overline{f}: \mathbb{Z}[M] \to A$ defined by  
$$\sum_{m \in M} \alpha_m\cdot m \mapsto \sum_{m \in M} \alpha_m f(m)$$ (with $\alpha_m \in \Z$) factors over the $(n+1)$'st power $I^{n+1}$ of the augmentation ideal $I \subseteq \Z[M]$. 
In other words, there is a canonical bijection
\begin{equation} \label{eq_corep}
 \Hom_{\leq n}(M , A) \xrightarrow{\simeq} \Hom_{\mathrm{Ab}}(\mathbb{Z}[M]/ I^{n+1} , A) .
\end{equation}
This fact is proven in \cite{Passi}, but let us give an argument. The augmentation ideal $I$ is generated additively by elements of the form $(m-1)$ with $m \in M$. Therefore $I^{n+1}$ is generated additively by elements of the form 
$$(m_0 - 1) \cdot \ldots \cdot (m_n -1) $$
with $m_i \in M$. A slightly bigger additive generating set for $I^{n+1}$ is then given by 
$$x \cdot  (m_0 - 1) \cdot \ldots \cdot (m_n -1) $$
with $m_i, x \in M$.
Using this fact we have to show that $f: M \to A$ is polynomial of degree at most $n$ precisely if $\overline{f}: \Z[M] \to A$ vanishes on these products for all $x,m_i \in M$. 
This follows from the following pair of observations:
\begin{itemize}
\item A map $f: M \to A$ is polynomial of degree at most $n$ precisely if for each sequence $m_0,...,m_n,x$ of elements in $M$ we have
$$
(D_{m_0}D_{m_1} ... D_{m_n} f)(x) = 0.
$$
\item There is an equality
$$
(D_{m_0}D_{m_1} ... D_{m_n} f)(x) = \overline{f}\Big(x \cdot (m_0 - 1) \cdot (m_1 - 1) \cdot \ldots \cdot (m_n -1)\Big).
$$
\end{itemize}
The first of these observations is the definition. The second observation follows inductively from the case $n =0$ which is obvious.
 
Now we can proceed to the proof of the theorem. By virtue of the natural bijection \eqref{eq_corep} it suffices to show that the map
$$
\Z[M]/ I^{n+1} \to \Z[M^+]/ I^{n+1}
$$
is an isomorphism of abelian groups. Both sides are actually rings and the map
in question is a map of rings. In order to construct an inverse ring map, it suffices to check that all elements $m \in M$
represent multiplicative units in $\Z[M]/ I^{n+1}$; 
however, this follows because $m-1$ is nilpotent. 
\end{proof}

The last result shows that group completion is universal with respect to
polynomial maps and not only for additive maps. 
From this, the  extended functoriality of $K_0$ readily follows, as in
\cite{Jou00}; we review the details below.

\begin{definition}[$K_0$ of additive $\infty$-categories]
For an additive $\infty$-category $\mathcal{A}$, the group $K_0(\calA)$ is the
group completion of the abelian monoid $\pi_0(\mathcal{A})$ of isomorphism classes of objects with
$\oplus$ as addition. Concretely $K_0(\calA)$ is the abelian group generated
from isomorphism classes of objects subject to the relation $[A] + [B] = [A
\oplus B]$. 
\end{definition}

Let $\Cat^\Add$ be the $\infty$-category of additive $\infty$-categories
and additive functors. Let $\mathrm{Ab}$ denote the ordinary category of abelian
groups. Then $K_0$ defines a functor $$
K_0: \Cat^{\Add} \to \mathrm{Ab}
$$

Let $\catap$ 
be the $\infty$-category of additive $\infty$-categories and polynomial functors
between them. 
The next result appears in the present form in \cite{Jou00} and (for modules
over a ring) in \cite{Dold72}.

\begin{proposition}[{\cite[Prop.~1.8]{Jou00}}]
\label{addK0func}
There is a functor $\widetilde{K}_0: \Cat^\Add_\poly \to \mathrm{Set}$ with a
transformation $\pi_0 \to \widetilde{K}_0$ such that the diagram
$$
\xymatrix{
\Cat^\Add\ar[d]  \ar[r]^-{K_0} & \Ab \ar[d] \\
\Cat^\Add_\poly \ar[r]^-{\widetilde{K}_0} & \mathrm{Set}
}
$$
commutes up to natural isomorphism. \end{proposition}
\begin{proof}
 For a polynomial functor $F: \calA \to \calB$ the map $\pi_0(\calA) \to
 \pi_0(\calB) \to K_0(\calB)$ is polynomial. Thus by \Cref{thm_passi}  it
 can be uniquely extended to a polynomial map $K_0(\calA) \to K_0(\calB)$. This gives the
 desired maps, and it is easy to see that they define a functor
 $\widetilde{K}_0: \catap \to \mathrm{Set}$. 
\end{proof}

\subsection{Stable $\infty$-categories}

In this section, we extend the results of the previous section to show that $K_0$ is functorial in polynomial functors of
stable $\infty$-categories; the strategy of proof is similar to that of
\cite{Dold72}. Recall that for stable $\infty$-categories, one has
the following definition of $K_0$, which only depends on the underlying
triangulated homotopy category. 

\begin{definition} 
Given a stable $\infty$-category $\mathcal{C}$, we define the group
$K_0(\mathcal{C})$ as the quotient of $K_0^{\mathrm{add}}(\mathcal{C})$ (i.e.,
$K_0$ of the underlying additive $\infty$-category) by the
relations $[X] + [Z] = [Y]$ for cofiber sequences $X \to Y \to Z$ in
$\mathcal{C}$. 
\end{definition}

\begin{proposition} 
\label{K0stable}
Let $F: \mathcal{C} \to \mathcal{D}$ be a polynomial functor between the stable
$\infty$-categories $\mathcal{C}, \mathcal{D}$. Then 
there is a unique polynomial map $F_*: K_0(\mathcal{C}) \to K_0(\mathcal{D})$ 
such that $F_*([X]) = [F(X)]$ for $X \in \mathcal{C}$. 
\end{proposition}

We have already seen an analog of this result for $K_0^{\mathrm{add}}$, the
$K_0$ of the underlying additive $\infty$-category (\Cref{addK0func}). 
The obstruction is to understand the interaction with cofiber sequences. For
this, we will need the following construction, and a general lemma about
simplicial resolutions.

\begin{cons}
\label{cechnerve}
Let $\mathcal{C}$ be a stable $\infty$-category. 
Suppose given a cofiber sequence $X' \to X \to X''$ in $\mathcal{C}.$
Then we form the \v{C}ech nerve of the map $X \to X''$. This constructs a
1-skeletal
simplicial object $A_\bullet$ in $\mathcal{C}$ of the form
\[ 
\xymatrix{
\dots \ar@<1.5pt>[r]\ar@<-1.5pt>[r]\ar@<4.5pt>[r]\ar@<-4.5pt>[r]& X ' \oplus X' \oplus X \ar@<0pt>[r]\ar@<3pt>[r]\ar@<-3pt>[r] & X' \oplus X \ar@<1.5pt>[r]\ar@<-1.5pt>[r] & 
X 
} 
.
 \]
Alternatively, we can consider this simplicial object as the two-sided bar
construction of the abelian group object $X' \in \mathcal{C}$ acting on $X$
(via $X' \to X$). 
We observe that each of the terms in the simplicial object, and each of the
face maps $d_i, i \geq 1$ depend only on the objects $X', X$ (and not on the
map $X' \to X$). Also, $A_\bullet$ is augmented over $X''$ and is a resolution
of $X''$. 
\end{cons}

\begin{lemma} 
\label{K0identify}
Suppose $\mathcal{C}$ is a stable $\infty$-category and $X^1_\bullet,
X^2_\bullet \in \fun(\Delta^{op}, \mathcal{C})$ are two simplicial objects such
that: 
\begin{enumerate}
\item Both $X^1_\bullet, X^2_\bullet$ are $d$-skeletal for some $d$.  
\item We have an identification $X^1_n \simeq X^2_n$ for each $n$. 
\item Under the above identification, the face maps $d_i, i \geq 1$ for both
simplicial objects are homotopic. 
\end{enumerate}
Then $|X^1_\bullet|, |X^2_\bullet|$ define the same class in $K_0(\mathcal{C})$.
\end{lemma} 
\begin{proof} 
This follows from the fact that $X^1_\bullet, X^2_\bullet$ have finite filtrations
whose associated gradeds are identified in view of the stable Dold-Kan correspondence. 
\end{proof} 

\begin{proposition} 
\label{quotientpoly}
Let $f: A \to A'$ be a polynomial map between abelian groups. Let $M \subset
A$ be an abelian submonoid. Suppose that for $ m \in M$ and $x \in A$, we have
$f(x + m) = f(x)$. Then for any $m' $ belonging to the subgroup $M'$ generated by
$M$ and for any $x \in A$, we have $f(x + m') = f(x)$ and $f$ factors over $A/M'$.
\end{proposition} 
\begin{proof} 
Fix $x \in A$. Consider the polynomial map $A \to A'$ sending $y \mapsto f(x+
y) - f(x)$. Since this vanishes for $y \in M$, it vanishes on the image of $M^+
\to A$ and the result follows. 
\end{proof}

\begin{proposition} 
Let $f: M \to A$ be a polynomial map from an abelian monoid $M$ to an abelian
group $A$. Let $N \subset M \times M$ be a submonoid which contains the
diagonal. Suppose that for each
$(m_1, m_2) \in N$, we have $f(m_1) = f(m_2)$. Then the unique
polynomial extension $f^+ : M^+ \to A$ factors over the quotient of $M^+$ by
the subgroup generated by $\{m_1 - m_2\}_{ (m_1, m_2) \in N }$. 
\end{proposition} 
\begin{proof} 
Note first that the collection 
$C = \{m_1 - m_2\}_{ (m_1, m_2) \in N } \subset M^+$ is a submonoid. 
We claim that for any $x \in M^+ $ and $c \in C$, we have $f^+(x) = f^+ ( x +
c)$.
Equivalently, for any $y \in M^+$, $f^+(y + m_1) = f^+(y + m_2)$. 
Since both are polynomial maps, it suffices to check this for $y \in M$, 
in which case it follows from our assumptions. 
Thus the function $f^+ : M^+ \to A$ is invariant under translations by elements
of $C$. Since $C$ is a monoid, it follows 
by Proposition~\ref{quotientpoly}
that $f^+$ is invariant under translations 
by elements of the subgroup generated by $C$. 
\end{proof}

\begin{proof}[Proof of \Cref{K0stable}] 
By Theorem~\ref{addK0func}, we have a natural polynomial map on additive $K$-theory 
\[K_0^{\ad}(\mathcal{C}) \xrightarrow{F_*^{\mathrm{add}}}  K_0^{\ad}(\mathcal{D})
,\]
such that 
$F^{\ad}_*([X]) = [F(X)]$ for $X \in \mathcal{C}$. 
It suffices to show the composite 
$K_0^{\ad}(\mathcal{C}) \xrightarrow{F_*^{\mathrm{add}}}  K_0^{\ad}(\mathcal{D})
\twoheadrightarrow K_0(\mathcal{D})$
factors through $K_0(\mathcal{C})$. 
To see this, recall that $K_0(\mathcal{C})$ is the quotient of
$K_0^{\ad}(\mathcal{C})$ (in turn the group completion of
$\pi_0(\mathcal{C})$) by the relations $[X]  = [X' \oplus X'']$ for each
cofiber sequence 
\begin{equation} \label{cofiber} X' \to X \to X'' .\end{equation}
The collection of such defines a submonoid of $\pi_0(\mathcal{C}) \times
\pi_0(\mathcal{C})$ containing the diagonal. 
To prove the assertion,  
 we need to show\footnote{Compare also \cite[Th.~A]{Jou00} for a related type
 of statement.} that if \eqref{cofiber} is a cofiber sequence
in $\mathcal{C}$, then 
\[ [F(X)] = [F(X' \oplus  X'')].  \]
To see this, we construct two simplicial objects $C^1_\bullet$ and
$C^2_\bullet$ as in Construction~\ref{cechnerve} such that: 
\begin{enumerate}
\item $C^1_\bullet, C^2_\bullet$ are identified in each degree $n$ with $X'
\oplus X''[-1]^{\oplus n}$ and the face maps $d_i, i
\geq 1$ are homotopic.  
\item We have $|C^1_\bullet| \simeq X' \oplus X''$ and $|C^2_\bullet| \simeq X$.
\item Both $C^1_\bullet, C^2_\bullet$ are 1-skeletal. 
\end{enumerate}
Namely, $C^1_\bullet$ is the \v{C}ech nerve of $X'  \xrightarrow{(\mathrm{id},
0)} X' \oplus X''$ while 
$C^2_\bullet$ is the \v{C}ech nerve of $X' \to X$. 
We then find that the simplicial objects $F(C^1_\bullet), F(C^2_\bullet)$ are
$n$-skeletal (if $F$ has degree $\leq n$) and the geometric realizations
are given by $F(X' \oplus X''), F(X)$ respectively. 
Moreover, $F(C^1_\bullet), F(C^2_\bullet)$ agree in each degree and the face maps
$d_i, i \geq 1$ are identified. 
By Lemma~\ref{K0identify}, it follows that their geometric realizations have the same class in $K_0$, as
desired. 
\end{proof} 

\section{The main result}
\subsection{The universal property of higher $K$-theory}

\newcommand{\catexk}{\mathrm{Cat}_{\infty, \kappa}^{\mathrm{perf}}}
\newcommand{\catexo}{\mathrm{Cat}_{\infty, \omega}^{\mathrm{perf}}}

Our first goal is to review the axiomatic approach to higher $K$-theory, and
its characterization. 
We will use the $K$-theory of stable
$\infty$-categories, as developed by \cite{BGT} and \cite{Baruniv}, following
ideas that go back to Waldhausen \cite{Wal85} and ultimately Quillen
\cite{Qui73a}. 

Throughout, we fix (for set-theoretic reasons) a regular cardinal $\kappa$. 
Recall that $\catex$ is compactly generated \cite[Cor.~4.25]{BGT}. 
Let $\catexk$ denote the subcategory of $\kappa$-compact objects.  

\newcommand{\catsm}{\mathrm{Cat}^{\mathrm{perf}}_{\mathrm{sm}}}
\newcommand{\funp}{\mathrm{Fun}^{\pi}}
\newcommand{\funpadd}{\mathrm{Fun}^{\pi}_{\mathrm{add}}}
\newcommand{\funadd}{\mathrm{Fun}^{\mathrm{add}}}
\newcommand{\thypk}{\mathrm{Thy}^{\mathrm{poly}}_{\kappa}}
\newcommand{\thypa}{\mathrm{Thy}^{\mathrm{poly,  \ add}}_{\kappa}}

\newcommand{\Sp}{\mathrm{Sp}}

\begin{definition}[Additive invariants] 
\begin{enumerate}
\item  
Let $\funp( \catexk, \mathcal{S})$ denote the $\infty$-category of finitely
product-preserving functors $\catexk \to \mathcal{S}$. 

\item
We say that $f \in\funp( \catexk, \mathcal{S})$ is \emph{additive} if $f$ is
grouplike\footnote{Note that for any $\mathcal{C} \in \catexk$,
$f(\mathcal{C})$ has the structure
of an $E_\infty$-monoid in $\mathcal{S}$ using the addition on $\mathcal{C}$;
the condition that $f$ is grouplike is that $f(\mathcal{C})$ is grouplike for
each $\mathcal{C}$.} and $f$ carries semiorthogonal decompositions in $\catexk$ to
products. 

\end{enumerate}
We let $\funpadd( \catexk, \mathcal{S}) \subset \funp(\catexk, \mathcal{S})$ be
the subcategory of additive invariants. 
This inclusion admits a left adjoint $(-)_{\mathrm{add}}$, called 
\emph{additivization}. 
\end{definition} 

The construction $\iota$ which carries $\mathcal{C} \in \catexk$ to its
underlying space (i.e., the nerve of the maximal sub $\infty$-groupoid) yields
an object of $\funp( \catexk, \mathcal{S})$. 
The construction of the algebraic $K$-theory space $K(-)$ 
yields an additive invariant, by Waldhausen's additivity theorem. 

\begin{theorem}[Compare \cite{BGT, Baruniv}] 
\label{univKtheory}
The $K$-theory functor $K: \catexk \to \mathcal{S}$ is the 
additivization of $\iota \in \funp( \catexk, \mathcal{S})$. 
\end{theorem} 

\begin{remark} 
As the results in \emph{loc.~cit.} are stated slightly differently (in
particular, $\kappa =\aleph_0$ is assumed), we briefly indicate how to deduce
the present form of 
\Cref{univKtheory}. 

To begin with, we reduce to the case $\kappa = \aleph_0$. 
Let $F = (\iota)_{\mathrm{add}}$ denote the additivization of $\iota$
considered as an object of $\fun^{\pi}(\catexk, \mathcal{S})$. 
We can also consider the additivization of $\iota$ 
considered as an object of 
$\fun^{\pi}(\catexo, \mathcal{S})$ and left Kan extend from $\catexo $ to
$\catexk$; we denote this by $F': \catexk \to \mathcal{S}$. 
By left Kan extension, we also have a map $\iota \to F'$ in $\fun^{\pi}(\catexk,
\mathcal{S})$. 

Now $F'$ is also an additive invariant, thanks to \cite[Prop.~5.5]{HSS17}. 
It follows that
we have maps in $\fun^{\pi}(\catexk, \mathcal{S})$ under $\iota$ from $F' \to F$ 
and $F \to F'$, using the universal properties. 
It follows easily (from the universal properties
in $\funp(\catexk,\mathcal{S})$ and $\funp(\catexo, \mathcal{S})$)
that the composites in both directions are the identity,
whence $F \simeq F'$. 

Thus, we may assume $\kappa = \aleph_0$ 
for the statement of \Cref{univKtheory}. For $\kappa = \aleph_0$, 
we have that $\funp_{\mathrm{add}}( \catexo, \mathcal{S})$ is the
$\infty$-category of $\mathrm{Sp}_{\geq 0}$-valued additive invariants in the
sense of \cite{BGT}, whence the result. 
\end{remark} 

We will also need a slight reformulation of the universal property, using a
variant of the definition of an additive invariant, which turns out to be
equivalent. 
In the following, we write $\mathrm{Fun}_{\mathrm{ex}}(-, -)$ denote the
$\infty$-category of exact functors between two stable $\infty$-categories. 
\begin{definition}[Universal $K$-equivalences] 
A functor $F: \mathcal{C} \to \mathcal{D}$ in $\catst$ is said to be a
\emph{universal $K$-equivalence} if there exists a functor $G: \mathcal{D} \to
\mathcal{C}$ such that 
\begin{equation} \label{Kequivdef} [G \circ F] = [\mathrm{id}_{\mathcal{C}}] \in
K_0 ( \fun_{\mathrm{ex}}(\mathcal{C},
\mathcal{C}))
, \quad [F \circ G] = 
 [\mathrm{id}_{\mathcal{D}}] \in K_0 ( \fun_{\mathrm{ex}}(\mathcal{D},
\mathcal{D}))
.\end{equation}
Equivalently, this holds if and only if for every $\mathcal{E}
\in \catst$, the natural map $\mathrm{Fun}_{\mathrm{ex}}(\mathcal{D},
\mathcal{E}) \to \mathrm{Fun}_{\mathrm{ex}}(\mathcal{C}, \mathcal{E})$ induces
an isomorphism on $K_0$. 
\end{definition} 
\begin{example} 
The shear map $\mathcal{C} \times \mathcal{C} \to \mathcal{C} \times
\mathcal{C}$, i.e., the functor $(X, Y) \mapsto (X \oplus Y, Y)$,
is a universal $K$-equivalence. 
If $\mathcal{C}$ admits a semiorthogonal decomposition into subcategories
$\mathcal{C}_1, \mathcal{C}_2$, then the projection $\mathcal{C} \to
\mathcal{C}_1 \times \mathcal{C}_2$ is a universal $K$-equivalence. 
\end{example}

\begin{proposition} 
A functor in $\funp( \catexk, \mathcal{S})$ is additive if and only if it
carries universal $K$-equivalences to equivalences. 
\end{proposition} 
\begin{proof} 
By the above examples, any object in 
$\funp( \catexk, \mathcal{S})$ which preserves universal $K$-equivalences is necessarily additive. 
Thus, it remains only to show that an additive invariant carries universal 
$K$-equivalences to equivalences. 
Note that an additive invariant $f: \catexk \to \mathcal{S}$ naturally lifts to
$\mathrm{Sp}_{\geq 0}$, since
it is grouplike. 
Moreover, given 
$\mathcal{C}, \mathcal{D} \in \catexk$, the map obtained by applying $f$,
\[ \pi_0( \fun_{\mathrm{ex}}(\mathcal{C}, \mathcal{D})^{\simeq}) \xrightarrow{f} 
\pi_0\mathrm{Hom}_{\mathrm{Sp}_{\geq 0}}( f(\mathcal{C}), f(\mathcal{D}))
\]
has the property that it factors through $K_0 (
\mathrm{Fun}_{\mathrm{ex}}(\mathcal{C}, \mathcal{D}))$: indeed, this follows
using additivity for $\mathcal{D}^{\Delta^1}$. 
This easily shows that $f$ sends universal $K$-equivalences to equivalences. 
\end{proof} 

We thus obtain the following result, showing that additivization is the
Bousfield localization at the universal $K$-equivalences. 

\begin{corollary} 
\label{univKislocal}
$\funp_{\mathrm{add}}( \catexk,\mathcal{S})$ is the Bousfield localization of
$\funp( \catexk, \mathcal{S})$ at the class of maps in $\catexk$ (via the
Yoneda embedding) which are universal $K$-equivalences.  \qed
\end{corollary}

\subsection{The universal property with polynomial functors}

In this subsection, we formulate the main technical result 
(Theorem~\ref{mainpolythm}) of the paper, which controls the additivization of a theory
functorial in polynomial functors. Throughout, we fix a regular
\emph{uncountable} cardinal
$\kappa$.

\begin{definition} 
 We let $\catpk$ denote the $\infty$-category whose objects are $\kappa$-compact
 idempotent-complete,
 stable $\infty$-categories
 and whose morphisms are polynomial functors between them. 

We consider the $\infty$-category 
$\funp( \catpk, \mathcal{S})$ of functors $\catpk \to \mathcal{S}$ which
preserve finite products. 
We say that an object $T \in \funp( \catpk, \mathcal{S})$ is \emph{additive} if
its restriction to $\catexk$ is additive. 
We let $\funp_{\mathrm{add}}(\catpk, \mathcal{S}) \subset \funp( \catpk,
\mathcal{S})$ denote the subcategory of additive objects. 
This inclusion admits a left adjoint $(-)_{\mathrm{addp}}$, called
\emph{polynomial additivization.}
\end{definition}

As an example, the underlying $\infty$-groupoid functor still defines 
a functor $\iota: \catpk \to \mathcal{S}$ which preserves finite products, and
hence an object of $\funp(\catpk, \mathcal{S})$. 
We now state the main technical result, which states that 
the polynomial additivization recovers 
the additivization when restricted to $\catexk$. 
This will be proved below in section~\ref{proofmainthm}.

\begin{theorem} 
\label{mainpolythm}
Let $T \in \funp( \catpk, \mathcal{S})$ and let $T_{\mathrm{addp}}$ denote its
polynomial additivization. 
Then the map $T \to T_{\mathrm{addp}}$, when restricted to $\catexk$, exhibits 
the restriction 
$T_{\mathrm{addp}}|_{\catexk}$ as the additivization of $T|_{\catexk}$. 
\end{theorem} 

A direct consequence of the theorem is a sort of converse: 
given any map $T \to T'$ such that the restricted transformation
exhibits $T'|_{\catexk}$ as the additivization of  $T|_{\catexk}$, then $T'$ is already the polynomial additivization. To see this simply note that $T'$ 
 is additive since this is only a condition on the restricted functor. Thus we get a map  $T_{\mathrm{addp}} \to T'$ which is, by the theorem, an equivalence when restricted to $\catexk$ and therefore an equivalence. 

Taking $T =\iota$ and using the universal property of $K$-theory (as in
\Cref{univKtheory}), we obtain the polynomial functoriality of $K$-theory
(\Cref{mainthm} from the introduction). 

\begin{corollary} 
There is a (unique) functor $\widetilde{K}: \catpk \to \mathcal{S}$ together with
a map $\iota \to \widetilde{K}$ in $\funp( \catpk, \mathcal{S})$, such that
the underlying map $\iota|_{\catexk} \to 
\widetilde{K}|_{\catexk}$ identifies $\widetilde{K}|_{\catexk}$ with $K$.
Moreover, $\widetilde{K}$ is
the polynomial additivization of $\iota$. 
\qed
\end{corollary}

\begin{remark} 
Proving such a result directly (e.g., by examining the $S_\bullet$-construction)
seems to be difficult. In fact, since the maps on $K$-theory spaces induced by polynomial functors are in general not loop maps they cannot be induced by maps of the respective $S_\bullet$-constructions. 
\end{remark}

\subsection{Generalities on Bousfield localizations}
We need
some preliminaries about 
strongly saturated collections, cf. \cite[Sec. 5.5.4]{Lur09}. 
\begin{definition} 
Let $\mathcal{E}$ be a presentable $\infty$-category. A \emph{strongly
saturated} class of maps  is a full subcategory of $\fun(\Delta^1, \mathcal{E})$
which is closed under colimits, base-changes, and compositions. 
\end{definition} 

\begin{cons}[Strongly saturated classes correspond to Bousfield localizations]
Given a set of maps in $\mathcal{E}$,  they generate a smallest strongly
saturated class. A strongly saturated class arising in this way is said
to be of \emph{small generation.} 

The class of maps in $\mathcal{E}$ that map to
equivalences under a Bousfield localization $\mathcal{E} \to \mathcal{E}'$ of
presentable $\infty$-categories is strongly saturated and of small
generation, and this in fact establishes a correspondence between accessible
localizations and strongly saturated classes of small generation \cite[Props.
5.5.4.15-16]{Lur09}. 
Specifically, given a set $S$ of maps, the Bousfield localization corresponding
to the strongly saturated class generating is the Bousfield localization whose
image consists of the $S$-local objects. 
To summarize, given a presentable $\infty$-category $\mathcal{E}$, we have a
correspondence between the following collections: 
\begin{itemize}
\item  Presentable $\infty$-categories $\mathcal{E}'$, equipped with fully
faithful right adjoints $\mathcal{E}' \to \mathcal{E}$ (so the left adjoint is
a localization functor). 
\item Strongly saturated classes of maps in $\mathcal{E}$ which are of small
generation. 
\item Accessible localization functors
$L: \mathcal{E} \to \mathcal{E}$. 
\end{itemize}
\end{cons}

\begin{proposition} 
\label{preserveclass}
Let $\mathcal{C}$ be a presentable $\infty$-category which is given as the
nonabelian derived $\infty$-category of a subcategory $\mathcal{C}_0\subset
\mathcal{C}$ closed under finite coproducts. Let $S$ be the strongly saturated
collection of maps in $\mathcal{C}$ generated by a subset $S_0 \subset
\fun(\Delta^1, \mathcal{C}_0)$. Suppose that
$S_0$ is closed under finite coproducts and contains the identity maps. Let $F: \mathcal{C} \to \mathcal{D}$ be
a functor which preserves sifted colimits and let $V$ be a strongly saturated
class in $\mathcal{D}$. If $F(S_0) \subset V$, then $F(S)
\subset V$. 
\label{stronglysaturatedprop}
\end{proposition} 
\begin{proof} 
Consider the collection $\mathscr{M}$ of maps $x \to y$ in $\mathcal{C}$ such that 
for every map $x \to x'$, the map
$F(x' \to y
\cup_x x')
)
\in V$. This collection $\mathscr{M}$ (in $\fun(\Delta^1, \mathcal{C})$) is clearly
closed under base-change, composition, and sifted colimits. 
Therefore, $\mathscr{M}$ is closed under all colimits and is in particular a
strongly saturated class. 

We claim that this collection $\mathscr{M}$ contains all of $S$; it suffices to
see that $S_0 \subset \mathscr{M}$. To see this, let $x_0 \to
y_0$ be a map in $S_0$. 
We need to see that the base-change of this map along a map $x_0 \to x_0'$
is carried by $F$ into $V$. 
Any map $x_0 \to x_0'$ can be written as a sifted
colimit of maps $x_0 \to x_0 \sqcup z$ for $z \in \mathcal{C}$, 
so one reduces to this case. 
Writing $z$ as a sifted colimit of objects in $\mathcal{C}_0$, 
we reduce to the case where $z = z_0 \in \mathcal{C}_0$. Then the assertion is
part of the hypotheses, so we obtain $(x_0 \to y_0) \in \mathscr{M}$ as desired. 
\end{proof}

\begin{corollary} 
\label{generalcorloc}
Let $\mathcal{A} , \mathcal{B}$ be $\infty$-categories admitting finite
coproducts.
Let $F_0: \mathcal{A} \to \mathcal{B}$ be a functor preserving finite coproducts,
inducing a cocontinuous functor 
$F: \mathcal{P}_\Sigma(\mathcal{A}) \to \mathcal{P}_\Sigma(\mathcal{B})$ with a
right adjoint $G: \mathcal{P}_\Sigma(\mathcal{B}) \to
\mathcal{P}_\Sigma(\mathcal{A})$ which preserves sifted colimits. 

Let $S_0$ be a class of maps in $\mathcal{A}$ which is closed under coproducts
and contains the identity maps. Let $T_0 = F(S_0) $ denote the
induced class of maps in $\mathcal{B}$; let $S, T$ be the induced strongly
saturated classes of maps in $\mathcal{P}_\Sigma(\mathcal{A}) \to
\mathcal{P}_\Sigma(\mathcal{B})$, and let $L_S, L_T$ be the associated
Bousfield localization functors. 
Suppose that the class of maps $G(T_0)  = GF(S_0)$ in 
$\mathcal{P}_\Sigma(\mathcal{A})$ 
belongs to the strongly saturated class generated by $S_0$. 

Then the functor 
$G : \mathcal{P}_\Sigma(\mathcal{B}) \to \mathcal{P}_\Sigma(\mathcal{A})  $
commutes with the respective localization functors. 
More precisely:

\begin{enumerate}
\item  For any $Y \in \mathcal{P}_\Sigma(\mathcal{B})$ which is $T$-local,
$G(Y)$ is $S$-local. 
\item
For any $Y' \in \mathcal{P}_\Sigma(\mathcal{B})$, the natural map 
$Y' \to L_T (Y')$ induces (by the property (1)) a map 
$L_S G(Y') \to G( L_T (Y'))$; this map is an equivalence. 
\item $G$ induces a functor $L_T \mathcal{P}_\Sigma( \mathcal{B}) \to L_S
\mathcal{P}_\Sigma(\mathcal{A})$ which commutes with limits and sifted
colimits, which  is right adjoint to the functor $L_T F:
L_S \mathcal{P}_\Sigma(\mathcal{A}) \to  L_T \mathcal{P}_\Sigma(\mathcal{B})$. 
\end{enumerate}\end{corollary} 
\begin{proof} 
Part (1) follows because $F$ (which preserves colimits) carries $S$ into $T$, so
the right adjoint 
necessarily carries $T$-local objects into $S$-local objects. 

By \Cref{stronglysaturatedprop}, $G$ carries the strongly saturated class $T$ in
$\mathcal{P}_\Sigma(\mathcal{B})$ 
into the strongly saturated class $S$ in $\mathcal{P}_\Sigma(\mathcal{A})$. 
Now in (2), the map $Y' \to L_T(Y')$ belongs to the strongly saturated class $T$, whence
$G(Y') \to G(L_T(Y'))$ belongs to the strongly saturated class $S$. Since the
target of this map is $S$-local, it follows that $L_S G(Y') \xrightarrow{\sim}
G(L_T(Y'))$. 
This proves part (2). 

For (3), we already saw in (1) that $G$ induces a functor $L_T
\mathcal{P}_\Sigma(\mathcal{B}) \to L_S \mathcal{P}_\Sigma(\mathcal{A})$, and
clearly this commutes with limits. 
It also commutes with sifted colimits since the functor $G:
\mathcal{P}_\Sigma(\mathcal{B}) \to \mathcal{P}_\Sigma(\mathcal{A})$ commutes
with sifted colimits and since $G$ carries the $L_T$-localization into the
$L_S$-localization. From this (3) follows. 
\end{proof} 
\subsection{Proof of \Cref{mainpolythm}}
\label{proofmainthm}
The proof 
of Theorem~\ref{mainpolythm} will require some more preliminaries. 
To begin with, we will need the construction of a universal target for a degree
$\leq n$ functor. 

\begin{cons}
Given $\mathcal{C} \in \catex$, we define the object $\Gamma_n \mathcal{C}
\in \catex$ such that we have a natural equivalence for any $\mathcal{D} \in
\catex$,
\[ \fun_{\mathrm{ex}}( \Gamma_n \mathcal{C}, \mathcal{D}) \simeq \fun_{\leq n}(\mathcal{C},
\mathcal{D}). \]
In particular, $\Gamma_n$ receives a degree $\leq n$ functor $\mathcal{C} \to
\Gamma_n \mathcal{C}$ and $\Gamma_n \mathcal{C}$ is universal for this
structure. 
Explicitly, $\Gamma_n \mathcal{C}$ is obtained by starting with the free
idempotent-complete stable
$\infty$-category on $\mathcal{C}$, i.e., compact objects in $\sp$-valued
presheaves on $\mathcal{C}$, and then forming the minimal exact localization such that
the Yoneda functor becomes $n$-excisive. 
\end{cons}

\begin{remark}[Some cardinality estimation] 
Recall again that $\kappa$ is assumed to be uncountable. 
If $\mathcal{C} \in \catexk$, then we claim that $\Gamma_n \mathcal{C} \in
\catexk$ for all $n \geq 0$. 
In fact, we observe that $\mathcal{C} \in \catexk$ if and only if $\mathcal{C}$ is
$\kappa$-compact as an object of $\mathrm{Cat}_\infty$; moreover, this holds if
and only if $\mathcal{C}$ has $<\kappa$ isomorphism classes of objects and the
mapping spaces in $\mathcal{C}$ are $\kappa$-small. 
\end{remark}

The crucial observation, for our purposes, is simply that $\Gamma_n$ behaves
relatively well with respect to semiorthogonal decompositions: it transforms
them into something that, while slightly more complicated, is very controllable
on $K$-theory.

\begin{proposition} 
\label{gammanunivK}
Let $F: \mathcal{C} \to \mathcal{C}'$ be a universal $K$-equivalence. 
 Then the 
map $\Gamma_n \mathcal{C} \to \Gamma_n \mathcal{C}'$ is a universal
$K$-equivalence. That is, for every $\mathcal{D} \in \catex$, the 
functor
\[ F^*: \fun_{\leq n}(\mathcal{C}', \mathcal{D}) \to
\fun_{\leq n}(\mathcal{C}, \mathcal{D}),   \]
induces an isomorphism on $K_0$. 
\end{proposition} 
\begin{proof} 
Let $G:\mathcal{C}' \to \mathcal{C}$ be a functor such that $F \circ G, G \circ
F$ satisfy \eqref{Kequivdef}.
It suffices
to show that the composite
$\fun_{\leq n}(\mathcal{C}, \mathcal{D}) \stackrel{G^*}{\to} 
\fun_{\leq n}(\mathcal{C}', \mathcal{D})
\stackrel{F^*}{\to} \fun_{\leq n}(\mathcal{C}, \mathcal{D})$
is the identity on $K_0$ (and the converse direction follows by symmetry). 

\newcommand{\ex}{\mathrm{ex}}
To see this, fix a functor $f \in 
\fun_{\leq n}(\mathcal{C}, \mathcal{D})$. We then have a degree $n$ functor
\[ f \circ \cdot : \fun_{\ex} (\mathcal{C} , \mathcal{C}) \to \fun_{\leq
n}(\mathcal{C},
\mathcal{D}), \quad \phi \mapsto f \circ \phi. \]
By Proposition~\ref{K0stable}, this induces a unique map on $K_0$. 
It follows that 
since $G \circ F, \mathrm{id}$ define the same class in $K_0( \fun_{\ex}
(\mathcal{C}' , \mathcal{C}))$, 
the functors $f \circ G \circ F, f$ define the same class in $K_0(\fun_{\leq n}(\mathcal{C},
\mathcal{D}))$. 
This shows precisely that $F^* \circ G^*$ induces the identity on $K_0$.
Similarly, $G^* \circ F^*$ induces the identity on $K_0$. 
This completes the proof. 
\end{proof}

\newcommand{\Res}{\mathrm{Res}}
\begin{proof}[Proof of Theorem~\ref{mainpolythm}] 
Consider the commutative diagram
\[ \xymatrix{
\funp_{\mathrm{add}}( \catpk, \mathcal{S}) \ar[d]^{\Res}  \ar[r] &
\funp(\catpk, \mathcal{S}) \ar[d]^{\Res}
\\
\funp_{\mathrm{add}}(\catexk, \mathcal{S}) \ar[r] &  \funp(\catexk, \mathcal{S})
}, \]
where the horizontal rows are the inclusions and the vertical arrows 
are given by restriction along $\catexk \subset \catpk$. 
Our goal is to show that when we reverse the horizontal arrows by replacing the
inclusion functors  by additivizations, the diagram still commutes.

This statement
fits into the setup of \Cref{generalcorloc}. 
Here we take $\mathcal{A} = (\catexk)^{op}$
and $\mathcal{B} = (\catpk)^{op}$, and $F_0$ to be the opposite of the inclusion
$\catexk \to \catpk$. 
Moreover, $S_0$ can be taken to be the class of 
universal $K$-equivalences in $\mathcal{A} = (\catexk)^{op}$. 
The local objects then correspond to the additive invariants (\Cref{univKislocal}).

Unwinding the definitions, we find that in order to apply \Cref{generalcorloc}, we now need to verify that 
if $\mathcal{C} \to \mathcal{D}$ is a universal $K$-equivalence in $\catexk$,
then the map 
in $\fun^{\pi}(\catexk, \mathcal{S})$ given by 
\[ \mathrm{Hom}_{\catpk}( \mathcal{D}, -) \to
\mathrm{Hom}_{\catpk}(\mathcal{C}, -) \]
induces an equivalence upon additivizations. 
Now by definition we have
\[ \mathrm{Hom}_{\catpk}( \mathcal{D}, -) = \varinjlim_{n}
\mathrm{Hom}_{\catex}(\Gamma_n \mathcal{D}, - ) , \]
and similarly for $\mathrm{Hom}_{\catpk}(\mathcal{C}, -) $. 
It therefore suffices to show that 

\[
\mathrm{Hom}_{\catex}( \Gamma_n \mathcal{D}, -) \to \mathrm{Hom}_{\catex}(
\Gamma_n \mathcal{C}, -)
\]
(as a map in $\fun^{\pi}(\catexk, \mathcal{S})$)
induces an equivalence on additivizations. 
This follows from \Cref{gammanunivK} 
and \Cref{univKislocal}, noting that  $\Gamma_n \mathcal{C}, \Gamma_n
\mathcal{D}$ belong to $\catexk$ since $\mathcal{C}, \mathcal{D}$ do and
$\kappa$ is uncountable.\end{proof} 

\section{Application to perfect $\mathbb{F}_p$-algebras}

In this section, we give a new proof of the classical and fundamental result of 
B\"okstedt calculating $\THH(\mathbb{F}_p)$, using elementary properties of
polynomial functors. 
As a warm-up, we prove a generalization of the following classical result. 
\begin{theorem}[{\cite{Hi81, Kra80}}] 
\label{Quillenthm}
Let $R$  be a perfect $\mathbb{F}_p$-algebra. Then $K_i(R) $ is a
$\mathbb{Z}[1/p]$-module for $i \geq 1$. 
\end{theorem}

Our approach to \Cref{Quillenthm} is very close to the proof in
\emph{loc.~cit.}; however, we avoid the use of the nilpotence of the
$\gamma$-filtration by using instead the following lemma. 

\begin{lemma} 
\label{reducedfunctorlem}
Let $R$ be an $\mathbb{F}_p$-algebra. 
Consider the exact category of degree $\leq p$, reduced functors $\proj(R) \to
\proj(R)$. 
This category contains the functors given by the Frobenius twist  $(-)^{(1)}$ and the $p$th tensor
power $m_p$. 
In $K_0$ of this category, we have $[(-)^{(1)}] -[m_p]$ is divisible by $p$. 
\end{lemma} 

\begin{proof} 
This congruence in fact holds in the category $\mathrm{Strict}_p$ of homogeneous degree $p$ strict
polynomial functors over $\mathbb{F}_p$ in the sense of \cite{FS97}, which act on
the category of finitely generated projective modules over any
$\mathbb{F}_p$-algebra.\footnote{In
fact, it is essentially the assertion that $\psi^p$ and Frobenius agree mod $p$.} 
By \cite[Th.~8.10]{HKT} (a calculation due to \cite{Gre80}), 
the class $K_0( \mathrm{Strict}_p)$ of any functor $F \in \mathrm{Strict}_p$ is
determined by the $\mathbb{G}_m^n$-character of $F( \mathbb{F}_p^n )$ for any $n
\geq p$; more precisely, $K_0( \mathrm{Strict}_p)$ is the abelian group of homogeneous
symmetric degree $p$ polynomials in the fundamental weights $\lambda_1, \dots,
\lambda_n$ of $\mathbb{G}_m^n$ (obtained by projection on each factor). 

Now the character of $(\mathbb{F}_p^n)^{(1)}$ 
is given by $\lambda_1^p + \dots + \lambda_n^p$ while the character of
$m_p(\mathbb{F}_p^n)$ is given by $(\lambda_1 + \dots + \lambda_n)^p$; since
these two homogeneous symmetric polynomials are congruent mod $p$, the result
follows. 
\end{proof}

Our generalization of 
\Cref{Quillenthm} is given by the following result. 
While it applies to $K$-theory  (by \Cref{mainthm}), it also applies to other
invariants such as $K( \mathrm{Fun}(\mathcal{C}, -))$ for any $\infty$-category
$\mathcal{C}$. 

\begin{proposition} 
\label{Zpinvertible}
Let $F \colon \catp \to \mathcal{S}$ be an additive invariant (in particular finitely product preserving). Then for any perfect $\mathbb{F}_p$-algebra $R$, we have that $\pi_i F(
\perf(R))$ is a $\mathbb{Z}[1/p]$-module for $i \geq 1$. 
\end{proposition} 
\begin{proof} 
It suffices to show that 
the Frobenius twist $(-)^{(1)}$ induces a map (for $i \geq 1$)
$\pi_i F(\perf(R)) \to \pi_i F(\perf(R))$ which is divisible by $p$. 
Note that an abelian group with an automorphism which is (as an endomorphism)
divisible by $p$ is necessarily a $\mathbb{Z}[1/p]$-module. 

The $p$th tensor power $m_p$ induces a map $F(\perf(R)) \to
F(\perf(R))$; this map factors as the composite of the 
diagonal map $F(\perf(R)) \to \prod_p F(\perf(R)) \xrightarrow{m} F(\perf(R))$,
where the second map comes from the multiplication map 
$\prod_p \perf(R) \to \perf(R)$. Since the multiplication map restricts to the
zero map when any factor is fixed at zero, it follows easily that 
$\prod_p F(\perf(R)) \xrightarrow{m} F(\perf(R))$ induces the zero map on
$\pi_i, i \geq 1$. 
Consequently, this factorization shows that $m_p$ 
induces the zero map on $\pi_i, i \geq 1$. 
But by 
\Cref{reducedfunctorlem}, $m_p$ is congruent modulo $p$ to the Frobenius twist
$(-)^{(1)}$, whence the claim. 
\end{proof}

\begin{corollary} 
If $R$ is a perfect $\mathbb{F}_p$-algebra, then the endomorphism $K$-groups
$K^{\mathrm{End}}_i(R)$ (i.e., the $K$-groups of the exact category of finitely
generated projective $R$-modules with an endomorphisms)
are $\mathbb{Z}[1/p]$-modules for $i > 1$. 
\end{corollary} 
\begin{proof} 
This follows using the identification $K^{\mathrm{End}}(R) = K(
\mathrm{Fun}(\Delta^1/\partial \Delta^1, \mathrm{Perf}(R))$ of
\cite{BGT16}, \Cref{Zpinvertible}, and the polynomial functoriality of
$K$-theory. 
\end{proof}

In the rest of this section, we carry out the application to
$\THH(\mathbb{F}_p)$. In particular, we reprove the following result,
cf.~\cite[Th.~5.5]{HeM97}. 

\newcommand{\TR}{\mathrm{TR}}
\begin{theorem} 
\label{TRFp}
We have $\TR(\mathbb{F}_p; p) = H\mathbb{Z}_p$. 
\end{theorem}

\Cref{TRFp} is proved in \cite{HeM97} as a consequence of the description of
$\THH(\mathbb{F}_p)$. 
However, thanks to the theory of topological Cartier modules of \cite{AN21}, 
\Cref{TRFp} is actually \emph{equivalent} to B\"okstedt's computation of $\THH(\mathbb{F}_p)$,
together with the cyclotomic Frobenius (and its explicit presentation as a
cyclotomic spectrum as in \cite[Sec.~IV.4]{NS18}). 
We describe this deduction below in \Cref{Bthm}.

We will directly deduce \Cref{TRFp} using elementary arguments with polynomial functors,
and, crucially, the result of  \cite{LM12} relating $\TR$ to a suitable
completion of the
$K$-theory of endomorphisms. By contrast, our argument will not rely on any
computational facts in stable homotopy theory such as the structure of the dual
Steenrod algebra, used in more classical proofs as in \cite{HN19}. In fact, these results can conversely be deduced from our result, i.e. the computation of the dual Steenrod algebra together with its Dyer-Lashof operations. 
Purely algebraic proofs of B\"okstedt's theorem have been given 
for the additive structure 
\cite{FLS94}; see
\cite{Kaledin, FonarevKaledin} for  recent purely algebraic proofs including the
multiplicative structure.

\begin{remark} 
\label{cofiberV}
As in \cite{AN21}, $\THH(\mathbb{F}_p)$ is the cofiber of the Verschiebung map
$V: \TR(\mathbb{F}_p; p)_{hC_p} \to \TR(\mathbb{F}_p; p)$. Since $\pi_0
\THH(\mathbb{F}_p) =\mathbb{F}_p$, it  is sufficient (in order to prove
\Cref{TRFp}) to verify that $\pi_i \TR(\mathbb{F}_p; p) =0$ for $i > 0$. 
We expect that this can be deduced from \Cref{Zpinvertible}: that is, we expect
that $\TR$ can be made functorial in polynomial functors. 
While we do not verify this statement here, it provided motivation for the argument. 
\end{remark} 

\begin{definition}[Endomorphism and cyclic $K$-theory] 
Given a finite-dimensional $\mathbb{F}_p$-vector space $V$, we let 
$\mathfrak{C}_V$ denote the stable $\infty$-category of module spectra over the
tensor algebra $T(V^{\vee})$ of $V^{\vee}/\mathbb{F}_p$. Equivalently, $\mathfrak{C}_V$ is the
derived $\infty$-category of the \emph{abelian} category 
$\mathfrak{C}_V^{\heartsuit}$
of $\mathbb{F}_p$-vector
spaces $X$ equipped with a map $X \to X \otimes V$.\footnote{We treat cyclic
$K$-theory using stable $\infty$-categories rather than abelian categories since our main result on polynomial
functors is expressed in terms of the former.} 
The presentable, stable $\infty$-category $\mathfrak{C}_V$ has a full
subcategory $\mathcal{C}_V$, consisting of those objects whose
underlying $\mathbb{F}_p$-module spectrum is perfect.

We let 
$K^{\mathrm{End}}( V)$
be the $K$-theory of  $\mathcal{C}_V \in \catst$. By the theorem of the heart
\cite{Bar15}, 
$K^{\mathrm{End}}(V)$
is
also the $K$-theory of the abelian category $\mathcal{C}_V^{\heartsuit}$ of finite-dimensional
$\mathbb{F}_p$-vector spaces $X$ equipped with a map $X \to X \otimes V$. 
There is a natural map
$K^{\mathrm{End}}( V) \to K(\mathbb{F}_p) = K^{\mathrm{End}}(0)$ and we write
$K^{\mathrm{cyc}}(V)$ (called the \emph{cyclic $K$-theory}) for the homotopy fiber. 
\end{definition} 

\begin{construction}[The constructions $R_n, S_n$] 
Given $V \in \vectw$, we let $R_n V$ denote the universal $\mathbb{F}_p$-vector
space with a 
degree $\leq n$ map $V \to R_n V$ (in the sense of 
\Cref{degreenmapgroup}\footnote{In fact, $R_n V$ is the quotient of the group
algebra $\mathbb{F}_p[V]$ by the $(n+1)$st power of the augmentation ideal,
cf.~the proof of \Cref{thm_passi}.})
and let $S_n V = (R_n V^{\vee})^{\vee}$. 
Note that we have a natural inclusion map 
$V \to S_n V$ for each $V \in \vectw$, which is dual to the natural surjection
$R_n V^{\vee} \to V^{\vee}$. Moreover, the constructions $V \mapsto R_n V, S_n V$ 
themselves define degree
$\leq n$ functors on $\vectw$. 
\end{construction}

\begin{construction}[Polynomial functors on cyclic $K$-theory] 
Let $f: \vectw \to \vectw$ be a degree $\leq n$ functor such that $f(0) = 0$. 
We can extend $f$ to a functor on all $\mathbb{F}_p$-vector spaces which
preserves filtered colimits. 

It follows that we obtain a natural polynomial functor
of additive categories
\begin{equation}  f: \mathfrak{C}_V^{\heartsuit} \to
\mathfrak{C}^{\heartsuit}_{S_n V} . \label{additivef} \end{equation}
Explicitly, given the pair
\( (X, X \to X \otimes V) ,  \) 
we carry this 
to the pair $( f(X), f(X) \to f(X) \otimes S_n V)$ obtained by applying $f$ to
the map $V^{\vee} \to \mathrm{Hom}_{\mathbb{F}_p}(X, X)$ 
(to obtain a degree $\leq n$ map $V^{\vee} \to \mathrm{Hom}_{\mathbb{F}_p}(f(X),
f(X))$)
and adjointing over. 
Taking left derived functors from the projective objects, we obtain a 
polynomial functor of stable $\infty$-categories
\( \mathfrak{C}_V \to \mathfrak{C}_{S_n V}  \)
extending \eqref{additivef}, and restricting to objects with underlying perfect
$\mathbb{F}_p$-module, we obtain a polynomial functor
of stable $\infty$-categories
\[ f: \mathcal{C}_V \to  \mathcal{C}_{S_n V}. \]

Consequently, we obtain a natural map on cyclic $K$-theory
\begin{equation} \label{keymap} f_* : K^{\mathrm{cyc}}(V) \to K^{\mathrm{cyc}}( S_n V).  \end{equation}
\end{construction} 

\begin{example}[Exact functors] 
\label{exactfunctors}
Suppose $f$ is actually an exact functor, considered as a functor of degree
$\leq n$ for some $n \geq 1$. 
In this case, the map \eqref{keymap} factors as 
$K^{\mathrm{cyc}}(V) \xrightarrow{f} K^{\mathrm{cyc}}(V) \to K^{\mathrm{cyc}}(S_n V)$,
where the first map arises from the exact functor $f: \mathcal{C}_V \to
\mathcal{C}_V$, and the second map arises by functoriality from the natural inclusion $V \subset
S_n V$. 
\end{example} 

\begin{example}[The tensor product] 
Suppose $f: \vectw \to \vectw$ is the functor given by the $n$th tensor power. 
Then we actually have a factorization
of $f_*$
\begin{equation} \label{factorizationn} K^{\mathrm{cyc}}(V) \xrightarrow{\Delta} \prod_n K^{\mathrm{cyc}}(V) \to K^{\mathrm{cyc}}(S_n
V) \end{equation}
since the tensor product factors as 
\[ \mathcal{C}_V \xrightarrow{\Delta} \prod_n \mathcal{C}_V \to
\mathcal{C}_{S_n V}.   \]
Indeed, this factorization holds on $\mathfrak{C}_V$. 
At the level of $\mathfrak{C}_V^{\heartsuit}$ (from which it follows in
general), we argue as follows: 
given an $n$-tuple of $\mathbb{F}_p$-vector spaces $\left\{X_i \right\}_{1 \leq
i \leq n}$ with maps $\phi_i: X_i \to X_i \otimes V$, we consider the adjoint
maps $V^{\vee} \to \mathrm{Hom}_{\mathbb{F}_p}(X_i, X_i)$ and use the tensor product
functoriality on $\mathbb{F}_p$-vector spaces. 
\end{example}

We now define $S_n V, K^{\mathrm{cyc}}(V)$ for an arbitrary animated
$\mathbb{F}_p$-vector space\footnote{Or simplicial $\mathbb{F}_p$-vector space,
cf.~\cite{CS} for a discussion of the terminology.} 
$V$ by animation (i.e., left Kan extension from finite-dimensional 
$\mathbb{F}_p$-vector spaces) of the above constructions. It follows that 
we still have a map 
\eqref{keymap} for an animated $\mathbb{F}_p$-vector space for each degree $\leq
n$ functor $f: \vectw \to \vectw$; moreover, this carries short exact sequences
of functors to sums of natural maps. 

Given a functor $G$ from animated $\mathbb{F}_p$-modules to spaces, 
let $P_m G$ denote the $m$th Goodwillie approximation, cf.~\cite[Sec.~6.1]{HA}
for a treatment in this language. 
The functor $V \mapsto S_n V$ is degree $\leq n$ as an endomorphism of animated
$\mathbb{F}_p$-vector spaces. 
Therefore, it follows that from the functor $f$, we obtain a map for each $m
\geq 1$ by taking Goodwillie approximations in \eqref{keymap},
\begin{equation} \label{Pmn} f_*: ( P_{mn}K^{\mathrm{cyc}}) (V)  \to (P_m K^{\mathrm{cyc}})( S_n
V). \end{equation}

\begin{proposition} 
Let $V$ be any animated $\mathbb{F}_p$-vector space. 
Then the map 
\eqref{Pmn} for $n = p$ and $f= \mathrm{id}$ is divisible by $p$ on homotopy groups $\pi_i, i \geq
1$. 
\end{proposition} 
\begin{proof} 
Via the above construction 
\eqref{keymap},
we have a map of abelian groups
\[ K_0 ( \mathrm{Fun}_{\leq p}( \vectw, \vectw)) \to 
\pi_0 \mathrm{Map}_{*}((P_{mp} K^{\mathrm{cyc}})(V), (P_m K^{\mathrm{cyc}})(S_p
V)).
\]
Our claim is that the identity functor induces a map 
\eqref{Pmn} which is divisible by $p$ on the higher homotopy groups. 
By the above observation and \Cref{reducedfunctorlem}, it suffices to show that the $p$th tensor power
functor $f_p$ has this property. 
But in fact, we have a factorization from taking approximations in
\eqref{factorizationn},
\[ (P_{mp}K^{\mathrm{cyc}})(V) \xrightarrow{\Delta} \prod_p
(P_{mp}K^{\mathrm{cyc}})(V) \to (P_m K^{\mathrm{cyc}})(S_p V), \]
and one checks (by comparing with the non-approximated case) that the 
second map restricts to zero when a single factor is zero. 
Thus, on homotopy groups, this map is zero on $\pi_i, i \geq 1$, whence the
claim. 
\end{proof} 

\begin{corollary} 
\label{maindivbyp}
For any $\mathbb{F}_p$-vector space $V$ and any $m \geq 1$, the
truncation map 
\begin{equation} \label{Truncmap} (P_{mp} K^{\mathrm{cyc}})(V) \to (P_m
K^{\mathrm{cyc}})(V)  \end{equation}
is divisible by $p$ on $\pi_i, i \geq 1$. 
\end{corollary} 
\begin{proof} 
If we consider the identity as a functor of degree $\leq p$, we 
we obtain a map
\[ \mathrm{id}_*:  
(P_{mp} K^{\mathrm{cyc}})(V) \to (P_m K^{\mathrm{cyc}})( S_pV) \]
which we have seen is divisible by $p$ on $\pi_i, i \geq 1$. 
Using \Cref{exactfunctors}, we see that this map 
is the composite of the truncation map 
\eqref{Truncmap} and the map induced by functoriality from $V \to S_p V$. 
Now for any $\mathbb{F}_p$-vector space $V$, the map $V \to S_p V$ is injective and
hence admits a splitting. Thus, the result follows. 
\end{proof}

\begin{proof}[Proof of \Cref{TRFp}] 
By the results of \cite{LM12}, 
$\TR(\mathbb{F}_p)$ is given by $\varprojlim_m (P_{m}
K^{\mathrm{cyc}})(\mathbb{F}_p)$. 
However, 
by \Cref{maindivbyp}, we thus find that $\pi_i \TR( \mathbb{F}_p)$ is uniquely 
$p$-divisible for $i \geq 1$, and hence vanishes since 
it is derived $p$-complete.
The same holds for the $p$-typical $\TR$, which is for $\mathbb{Z}_{(p)}$-algebras a summand of the `big' $\TR$, whence the result (cf.~also
\Cref{cofiberV}). 
\end{proof}

By the results of \cite{AN21}, the entire structure of
$\THH(\mathbb{F}_p)$ as a cyclotomic spectrum is determined by the knowledge of
$\TR(\mathbb{F}_p; p)$; the Frobenius and Verschiebung operators are forced for
truncatedness reasons. In particular, we include the deduction of B\"okstedt's
theorem in its original form. 
In the following, we freely use the language of \cite{NS18, AN21}, in particular
the $\mathbb{T}$-equivariant Frobenius and  Verschiebung on $\TR(-; p)$ and the
$\mathbb{T}$-equivariant cyclotomic Frobenius on $\THH(\mathbb{F}_p)$. 

\begin{corollary}[B\"okstedt] 
\label{Bthm}
We have $\THH(\mathbb{F}_p)_* = \mathbb{F}_p[\sigma]$ for $|\sigma| = 2$. 
\end{corollary} 
\begin{proof} 
By \cite{AN21} we have a pullback of ring spectra
\[
\xymatrix{
\TR(\mathbb{F}_p; p) \ar[d]\ar[r] & \THH(\mathbb{F}_p) \ar[d] \\
\TR(\mathbb{F}_p; p)^{hC_p} \ar[r] & \TR(\mathbb{F}_p; p)^{tC_p}  \ .
}
\]
Thus since $\TR(\mathbb{F}_p; p) = H \Z_p$ we conclude that the left vertical map has
$(-1)$-truncated cofibre, and hence the right vertical map is an isomorphism on
non-negative homotopy groups, which shows the claim. In fact, it shows that $\THH(\mathbb{F}_p) \simeq \tau_{\geq 0} (\mathbb{Z}_p^{tC_p})$ as (equivariant) $\mathbb{E}_\infty$-rings, which was deduced in \cite{NS18} from B\"okstedt's theorem.
\end{proof} 
Kaledin \cite{Kaledin} and Fonarev--Kaledin \cite{FonarevKaledin} also give
purely algebraic proofs of B\"okstedt periodicity. 
We expect our arguments should be related to those of \cite{Kaledin,
FonarevKaledin};
however, we do not use the language of trace theories, and the ingredients used
here seem to be slightly different. 

\bibliographystyle{amsalpha}
\bibliography{polynomial}

\end{document}